\documentclass[lettersize, 11pt]{article}
\usepackage{amsmath, color, subeqnarray, amsfonts, tikz, url, hyperref}
\usetikzlibrary{arrows}

\newtheorem{theorem}{Theorem}

\newtheorem{proposition}{Proposition}
\newtheorem{remark}{Remark}

\definecolor{darkblue}{rgb}{0.0,0.0,0.0}
\hypersetup{colorlinks,breaklinks,
            linkcolor=darkblue,urlcolor=darkblue,
            anchorcolor=darkblue,citecolor=darkblue}

\def\B{{\mathbb{B}}}
\def\R{{\mathbb{R}}}
\def\Z{{\mathbb{Z}}}

\def\Cupper{{\overline C}}
\def\Clower{{\underline C}}
\def\Vupper{{\overline V}}

\def\fred{\textcolor{black}}

\allowdisplaybreaks

\begin{document}

\title{ {\bf Convex Hulls for the Unit Commitment Polytope}\\[3mm]
\author{\normalsize \bf Kai Pan and Yongpei Guan
\\ \\
{\small Department of Industrial and Systems Engineering}\\
{\small University of Florida, Gainesville, FL 32611}\\
{\small Emails: kpan@ufl.edu; guan@ise.ufl.edu}\\
} }
\date{\small January 31, 2017}

\maketitle

\vspace{-0.3in}

\begin{abstract}
\setlength{\baselineskip}{18pt}
In this paper, we consider the polyhedral structure of the unit commitment polytope. In particular, we provide the convex hull results for the problem under the following different settings: 1) the convex hulls for the integrated minimum-up/-down time and ramping polytope under the general $T$ time period setting in which the ramping rate equals to the gap between the generation upper and lower bounds and equals to half of the gap between the generation upper and lower bounds, respectively, 2) the convex hull for the integrated minimum-up/-down time and ramping-up polytope for the problem under the general $T$ time period setting, and 3) the convex hull for the integrated minimum-up/-down time and ramping-down polytope for the problem under the general $T$ time period setting.
\vspace{0.10in}

\end{abstract}

\setlength{\baselineskip}{22pt}

\section{Introduction} \label{sec:introduction}
{To describe the polytope for each generator, we let} $T$ be the number of time periods {for the whole operational} horizon, $L$ ($\ell$) be the minimum-up (-down) time limit of the generator, $\Cupper$ ($\Clower$) be its generation upper (lower) bound when it is online, $\Vupper$ be its start-up/shut-down ramp rate, and $V$ be its ramp-up/-down rate in the stable generation {region}. In addition, we let $(x,y,u)$ be {the decision} variables to represent {the generator's} status, in which continuous variable $x$ represents the generation amount, binary variable $y$ represents {the generator's} online/offline status (i.e., $y_t = 1$ means the generator is online at $t$ and $y_t = 0$ otherwise), and binary variable $u$ represents whether the generator starts up or not (i.e., $u_t = 1$ means the generator starts up at $t$ and $u_t = 0$ otherwise). The corresponding integrated minimum-up/-down {time} and ramping polytope can be described as follows:
\begin{subeqnarray}
& P := \Bigl\{ (x, y, u) \in \R_+^{T} \times \B^{T} \times \B^{T-1} :  & \sum_{i = t- L +1}^{t} u_i \leq y_t, \ \forall t \in [L + 1, T]_{\Z}, \slabel{eqn:p-minup} \\
&& \sum_{i = t- \ell +1}^{t} u_i \leq 1 - y_{t- \ell}, \ \forall t \in [\ell + 1, T]_{\Z}, \slabel{eqn:p-mindn} \\
&& y_t - y_{t-1} - u_t \leq 0, \ \forall t \in [2, T]_{\Z}, \slabel{eqn:p-udef} \\
&&  - x_t + \Clower y_t \leq 0, \ \forall t \in [1, T]_{\Z}, \slabel{eqn:p-lower-bound} \\
&& x_t - \overline{C} y_t \leq 0, \ \forall i \in [1, T]_{\Z}, \slabel{eqn:p-upper-bound} \\
&& x_t - x_{t-1} \leq V y_{t-1} + \Vupper (1 - y_{t-1}), \ \forall t \in [2, T]_{\Z}, \slabel{eqn:p-ramp-up} \\
&& x_{t-1} - x_t \leq V y_t + \Vupper (1 - y_t), \ \forall i \in [2, T]_{\Z} \slabel{eqn:p-ramp-down} \Bigr\},
\end{subeqnarray}
where constraints \eqref{eqn:p-minup} and \eqref{eqn:p-mindn} describe the minimum-up and minimum-down time limits \cite{lee2004min, rajan2005minimum}, respectively (i.e., if the generator starts up at {time} $t-L+1$, it should keep online in the following $L$ consecutive time periods until {time} $t$; if the generator shuts down at time $t-\ell+1$, it should keep offline in the following $\ell$ consecutive time periods until {time} $t$), constraints \eqref{eqn:p-udef} describe the logical relationship between $y$ and $u$, constraints \eqref{eqn:p-lower-bound} and \eqref{eqn:p-upper-bound} describe the generation lower and upper bounds, and constraints \eqref{eqn:p-ramp-up} and \eqref{eqn:p-ramp-down} describe the generation ramp-up and ramp-down rate limits.
Note here that, in our polytope description, there is no start-up decision corresponding to the first-time period. In this way, the derived strong valid inequalities can be applied to each time period and can be used recursively. 
Meanwhile, considering the physical characteristics of {a thermal} generator, {without loss of generality, we can assume} $\Clower < \Vupper < \Clower + V$ and $\Cupper - \Clower - V \geq 0$. 
For notation convenience, we define $\epsilon$ as an arbitrarily small positive real number and {$[a,b]_{\Z}$} as the set of integer numbers between integers $a$ and $b$, i.e., $\{ a, a+1, \cdots, b \}$ {with} $[a,b]_{\Z} = \emptyset$ if $a > b$. Finally, we let conv($P$) represent the convex hull {description} of $P$. In addition, we consider the integrated minimum-up/-down time and ramping-up polytope as follows:
\begin{eqnarray}
P^{U} := \Bigl\{ (x, y, u) \in \R_+^{T} \times \B^{T} \times \B^{T-1} : \eqref{eqn:p-minup} - \eqref{eqn:p-ramp-up} \Bigr\}. \nonumber
\end{eqnarray}
Similarly, we consider the integrated minimum-up/-down time and ramping-down polytope as follows:
\begin{eqnarray}
P^{D} := \Bigl\{ (x, y, u) \in \R_+^{T} \times \B^{T} \times \B^{T-1} : \eqref{eqn:p-minup} - \eqref{eqn:p-upper-bound}, \eqref{eqn:p-ramp-down} \Bigr\}. \nonumber
\end{eqnarray}
Based on this description, we provide the convex hull descriptions in the following sections.

\section{Convex Hull Description conv($P$) for Two Special Cases with General $T$ Time Periods}
In this section, we first consider the case in which $V = \overline{C}-\Clower$. Under this setting, we derive the convex hull description for conv($P$). 


\begin{proposition} \label{prop:x_t-intg_k=1}
When $V = \overline{C}-\Clower$, the following inequality 
\begin{equation}
x_t \leq \Vupper y_t + (\overline{C} -\Vupper)\left(y_s-\sum_{i=0}^ju_{s-i}\right), \ \forall t \in [1, T],   \label{newineq}
\end{equation}
where $j=\min\{1, L-1, s-2, s-t\}$ and $s=\min\{t+1, T\}$, is valid and facet-defining for conv($P$). 
\end{proposition}
\begin{proof}
The validity can be easily proved by discussing the possible values of $y_s$. The facet-defining proof can be done by creating feasible and linear independent points. The details are shown in Appendix A.1.
\end{proof}
\begin{theorem} \label{thm:intg_k=1}
When $V = \overline{C}-\Clower$, the convex hull description for conv($P$) can be described as follows:
\begin{eqnarray}
& Q := \Bigl\{ & (x, y, u) \in \R^{3T-1}:  \eqref{eqn:p-minup} - \eqref{eqn:p-lower-bound}, \eqref{newineq}, \nonumber \\
&& u_t \geq 0, \ \forall t \in [2,T]_{\Z} \label{eqn:u-pos} \Bigr\}.
\end{eqnarray}
\end{theorem}
\begin{proof}
We can easily verify that polytope $Q$ is full-dimensional. We first characterize the extreme points of conv($P$) then show that each extreme point of conv($P$) satisfies $3T-1$ inequalities in $Q$ at equation, and finally conclude that the $Q=$conv$(P)$ due to Proposition \ref{prop:x_t-intg_k=1}.

Based on the minimum-up/-down time restrictions, we can observe that the generator will be online for certain time periods, offline for certain time periods, online for certain time periods, etc. For notation convenience, we assume there are $R$ online intervals (with $t^k_1$ and $t^k_2$ representing the starting and ending time unit for the $k$th ($1 \leq k \leq R$) online interval) for a $T$ time horizon problem. We can easily observe that all extreme points satisfy the following conditions (which can be easily proved by a contradiction method that constructs two feasible points to make this point as the linear combination of these two constructed points). 
\begin{itemize}
\item[1.] $x_{t^k_1}, x_{t^k_2} \in \{\Clower, \Vupper \}$ for each $k: 1\leq k \leq R$.
\item[2.] $x_t \in \{\Clower, \Cupper\}$ for each time period $t$ such that $t^k_1 < t < t^k_2$ for some $k: 1\leq k \leq R$.
\item[3.] $x_t =0$ for each other time period $t$.  
\end{itemize}
Now we can find the $3T-1$ tight inequalities corresponding to each extreme point as follows.
For an extreme point $z=(x_1, x_2, \ldots, x_T; y_1, y_2, \ldots, y_T; u_2, \ldots, u_T)$, if $x_t \in \{0, \Clower\},$ pick $x_t \geq \Clower y_t$, else $x_t \in \{\Vupper, \Cupper\},$ pick inequality~\eqref{newineq}. Thus, here we pick $T$ inequalities based on the value of $x$ part of this extreme point $z$. Now we prove why inequality~\eqref{newineq} is tight corresponding to this extreme point. We discuss this based on the following several conditions:
\begin{itemize}
\item[1)] If $x_t = \Vupper$, then $y_t=1$. From extreme point conditions, we essentially have $u_t=1$. 
\begin{itemize}
\item[(1)] If $s=T,$ i.e., $t=T$, then $\eqref{newineq} \Rightarrow x_t \leq \Vupper y_t + (\Cupper-\Vupper) (y_t - u_t)$ (which is easy to be verified that this is tight as $y_t=u_t=1$ and $x_t=\Vupper$.)
\item[(2)] If $s=t+1$, i.e., $s \leq T-1$, $\eqref{newineq} \Rightarrow x_t \leq \Vupper y_t + (\Cupper-\Vupper) (y_{t+1}-\sum_{i=0}^{\min\{1, L-1, t-1\}} u_{t+1-i}).$
\begin{itemize}
\item[(i)] If $y_{t+1}=0,$ i.e., $u_{t+1}=0$, then $u_t$ does not exist in $\sum_{i=0}^{\min\{1, L-1, t-1\}} u_{t+1-i}$. Therefore, $\eqref{newineq}$ is tight. 
\item[(ii)] If $y_{t+1}=1,$ i.e., $u_{t+1}=0$ and $u_t=1$, then $\eqref{newineq}$ is tight. 
\end{itemize}  
\end{itemize}
\item[2)] If $x_t=\Cupper,$ i.e., $y_{t-1} = y_t = y_{t+1} =1, u_t = u_{t+1} =0$, then it is easy to see that $\eqref{newineq}$ is tight. 
\end{itemize}

Besides these $T$ inequalities, another $2T-1$ tight inequalities can be easily picked from \eqref{eqn:p-minup}-\eqref{eqn:p-lower-bound} and \eqref{eqn:u-pos} based on the values of the $y$ and $u$ parts of this extreme point $z$, since \eqref{eqn:p-minup}-\eqref{eqn:p-lower-bound} and \eqref{eqn:u-pos} already construct the convex hull of the minimum-up/-down time polytope \cite{rajan2005minimum}.
\end{proof}

\begin{remark}
Note here that for the case $V < \overline{C}-\Clower$, the original formulation provides the convex hull description.  
\end{remark}

Next, we consider the case in which $\Cupper=\Clower+2V$. Under this setting, we derive the convex hull description for conv($P$). Without loss of generality, here we let $\Vupper=\Clower$ and then derive the following strong valid inequalities for conv($P$). 

\begin{proposition} \label{prop:x_t-intg-k=2}
For each $t \in [1,T]_{\Z}$, $S \subseteq \{\min \{t+2,T\}\}$, the inequality
\begin{align}
x_t & \leq \Vupper y_t + V \sum_{i \in S \cup \{\min\{t+1,T\}\} } (d_i - i) (y_i - \sum_{j=0}^{\min \{L-1,i-2\}} u_{i-j})+ V \sum_{j=0}^{\min \{L-1,t-2\}} j u_{t-j} \label{eqn:x_t-intg-k2}
\end{align}
is valid and facet-defining for conv($P$), where for each $i \in [t+1, t+2]_{\Z}$, $d_i = \min \{a \in S \cup \{t+3\}: a > i \}$.
\end{proposition}
\begin{proof}
See Appendix A.2.
\end{proof}

\begin{proposition} \label{prop:ramp-intg-k=2}
For each $t \in [1,T-1]_{\Z}$, the inequalities
\begin{align}
x_{t+1}-x_t & \leq \Vupper y_{t+1} - \Clower y_t + V (y_s - \sum_{i=0}^{j_1}u_{s-i}) \label{eqn:ru-intg-k2} \\
x_t-x_{t+1} & \leq \Vupper y_{t} - \Clower y_{t+1} + V (y_{t+1} - \sum_{i=0}^{j_2}u_{t+1-i}) \label{eqn:rd-intg-k2} 
\end{align}
are valid and facet-defining for conv($P$), where $s=\min\{t+2,T\}$, $j_1=\min\{s-2,1,L-1,s-t-1\}$, and $j_2=\min\{1,t-2,L-1\}$.
\end{proposition}
\begin{proof}
See Appendix A.3.
\end{proof}

\begin{theorem} \label{thm:intg_k=2}
When $\overline{C}=\Clower+2V$ and $\Vupper=\Clower$, the convex hull description for conv($P$) can be described as follows:
\begin{eqnarray}
& Q := \Bigl\{ (x, y, u) \in \R^{3T-1}:  \eqref{eqn:p-minup} - \eqref{eqn:p-lower-bound}, \eqref{eqn:u-pos}, \eqref{eqn:x_t-intg-k2}, \eqref{eqn:ru-intg-k2}, \eqref{eqn:rd-intg-k2} \Bigr\}. \nonumber
\end{eqnarray}
\end{theorem}
\begin{proof}
The proof is similar to that for Theorem \ref{thm:intg_k=1} by characterizing the extreme points of conv($P$) and showing each corresponding extreme point satisfies $3T-1$ inequalities in $Q$. We show the details in Appendix A.4.
\end{proof}

\begin{remark}
Inequalities~\eqref{newineq} and \eqref{eqn:x_t-intg-k2} - \eqref{eqn:rd-intg-k2} are different from any inequalities in the convex hull descriptions of the separate up and down polytope, i.e., the following $\mbox{conv}(P^U)$ and $conv(P^D)$.
\end{remark}

\section{Integrated Minimum-Up/-Down Time and Ramping-Up (or-Down) Polytope Convex Hull Results} \label{sec:valid-ineq}
In this section, we derive the convex hull results for the integrated minimum-up/-down time and ramping-up polytope, and the integrated minimum-up/-down time and ramping-down polytope, respectively. For this study, we do not restrict $V = \overline{C}-\Clower$ or $V = (\overline{C}-\Clower)/2$.

\begin{proposition} \label{prop:x_t-up-exp}
For each $t \in [1, T]_{\Z}$, $m \in [0, \min \{[t-L-1]^+, (\Cupper-\Vupper)/V, \fred{[\lfloor (\Cupper-\Clower)/V \rfloor - L]^+} \}]_{\Z}$, $S \subseteq [t-m+1, t-1]_{\Z}$, the inequality
\begin{align}
x_t & \leq \Vupper y_t + (L-1) V (y_t - \sum_{j=0}^{\min \{L-1,t-2, \fred{\kappa}\}} u_{t-j}) + V \sum_{i \in S \cup \{t\} } (i - d_i) (y_i - \sum_{j=0}^{\min \{L-1,i-2,\fred{\kappa}\}} u_{i-j}) \nonumber \\
 & + (\Cupper - \Vupper - (m+L-1) V) (y_{t-m} - \sum_{j=0}^{\min\{L-1,t-m-2,\fred{\kappa}\}} u_{t-m-j}) + V \sum_{j=0}^{\min \{L-1,t-2,\fred{\kappa}\}} j u_{t-j} \label{eqn:x_t-up-exp}
\end{align}
is valid and facet-defining for conv($P^U$), where \fred{$\kappa = \lceil (\Cupper-\Clower)/V \rceil - 1$} and for each $i \in [t-m+1, t]_{\Z}$, $d_i = \max \{a \in S \cup \{t-m\}: a < i \}$ and if $m=0$, then $d_t=t$.
\end{proposition}
\begin{proof}
See Appendix B.1.
\end{proof}

\begin{proposition} \label{prop:x_t-down-exp}
For each $t \in [1, T]_{\Z}$, $m \in [\min \{[T-t-1]^+, L-1, \fred{(\Cupper-\Vupper)/V}\}, \min \{[T-t-1]^+, (\Cupper-\Vupper)/V \}]_{\Z}$, $S \subseteq [t+L+1, t+m]_{\Z}$, the inequality
\begin{align}
x_t & \leq \Vupper y_t + V \sum_{i=1}^{\min\{m,L-1\}} (y_{t+i} - \sum_{j=1}^i u_{t+j}) + V \sum_{i \in S \cup \{t+L\}} (d_i - i) (y_i - \sum_{j=0}^{\min\{m,L-1\}} u_{i-j})  \nonumber \\
 & + (\Cupper - \Vupper - m V) (y_{t+m+1} - \sum_{j=0}^{\min\{m,L-1\}} u_{t+m+1-j}) \label{eqn:x_t-down-exp}
\end{align}
is valid and facet-defining for conv($P^D$), where for each $i \in [t+L, t+m]_{\Z}$, $d_i = \min \{a \in S \cup \{t+m+1\}: a > i \}$ and if $m \leq L-1$, then $d_{t+L}=t+L$. Meanwhile, we let $y_{T+1}=y_T$ and $u_{T+1}=0$.
\end{proposition}
\begin{proof}
See Appendix B.2.
\end{proof}

\begin{proposition} \label{prop:ru-2-exp}
For each $t \in [2, T]_{\Z}$, $m \in [1, \min \{t-1, \fred{\lceil (\Cupper-\Clower)/V \rceil - 1} \}]_{\Z}$, $S_0 \subseteq [t-m+L, t-1]_{\Z}$, $S = S_0 \cup \{t\}$, $q = \min \{a \in S\}$, $\delta = \min \{L-1,m-1\}$, the inequality
\begin{align}
 x_t - x_{t-m} & \leq \Vupper y_t - \Clower y_{t-m} + V \sum_{i \in S \setminus \{t-m+L\} } (i - d_i) (y_i - \sum_{j=0}^{\delta} u_{i-j}) \nonumber \\
 & + \delta V (y_t - \sum_{j=0}^{ \delta} u_{t-j}) + (\Clower + V-\Vupper) (y_q - \sum_{j=0}^{\delta} u_{q-j}) + V \sum_{j=0}^{\delta} j u_{t-j}, \label{eqn:ru-2-exp}
\end{align}
is valid and facet-defining for conv($P^U$), where for each $i \in S$, $d_i = \max \{a \in S \cup \{t-m+L\}: a < i \}$ and if $m \leq L$, then $d_t=t$.
\end{proposition}
\begin{proof}
See Appendix B.3.
\end{proof}

\begin{proposition} \label{prop:rd-2-exp}
For each $t \in [1, T-1]_{\Z}$, $m \in [1, \min \{T-t, \fred{\lceil (\Cupper-\Clower)/V \rceil - 1} \}]_{\Z}$, $S_0 = [t+L+1, t+m]_{\Z}$, $S = S_0 \cup \{t+L\}$, $q = \min \{a \in S\}$, $\delta = \min \{L-1,m-1\}$, the inequality
\begin{align}
x_t - x_{t+m} & \leq \Vupper y_t - \Clower y_{t+m} + V \sum_{i=1}^{\delta} (y_{t+i} - \sum_{j=1}^{i} u_{t+j}) + V \sum_{i \in S \setminus \{t+m\}} (d_i - i) (y_i - \sum_{j=0}^{\delta} u_{i-j})  \nonumber \\
 & + (\Clower + V - \Vupper) (y_{q} - \sum_{j=0}^{ \delta } u_{q-j}) \label{eqn:rd-2-exp}
\end{align}
is valid and facet-defining for conv($P^D$), where for each $i \in S \setminus \{t+m\}$, $d_i = \min \{a \in S \cup \{t+m\}: a > i \}$ and if $m \leq L$, then $d_{t+L}=t+L$.
\end{proposition}
\begin{proof}
See Appendix B.4.
\end{proof}

We denote $\alpha_1 = \max \{n \in [1,T]_{\Z}: \Clower + nV \leq \Cupper\}$, $\alpha_2 = \max \{n \in [1,T]_{\Z}: \Vupper + nV \leq \Cupper\}$, and $\mathcal{Q} = \{0, (\Clower+nV)_{n=0}^{\alpha_1}, (\Vupper+nV)_{n=0}^{\alpha_2}, (\Cupper-nV)_{n=0}^{\alpha_1}\}$. Note here that $\alpha_2 \leq \alpha_1 \leq T$ because $\Vupper \geq \Clower$.

\begin{proposition} \label{prop:duc_ext_point}
For any extreme point $(\bar{x}, \bar{y}, \bar{u})$ of conv($P^U$) or conv($P^D$), $\bar{x}_t \in \mathcal{Q}$ for all $t \in [1,T]_{\Z}$.
\end{proposition}
\begin{proof}
A more general conclusion is provided in \cite{pan2016polynomial1}. The details are shown in Appendix B.5.
\end{proof}

\begin{theorem}
The convex hull description for the integrated minimum-up/-down time and ramping-up polytope is
\begin{eqnarray}
\mbox{conv}(P^U) = Q^U := \Bigl\{ (x, y, u) \in \R^{3T-1}:  \eqref{eqn:p-minup} - \eqref{eqn:p-lower-bound}, \eqref{eqn:u-pos}, \eqref{eqn:x_t-up-exp}, \eqref{eqn:ru-2-exp} \Bigr\}. \nonumber 
\end{eqnarray}
The convex hull description for the integrated minimum-up/-down time and ramping-down polytope is
\begin{eqnarray}
 conv(P^D) = Q^D := \Bigl\{ (x, y, u) \in \R^{3T-1}:  \eqref{eqn:p-minup} - \eqref{eqn:p-lower-bound}, \eqref{eqn:u-pos}, \eqref{eqn:x_t-down-exp}, \eqref{eqn:rd-2-exp} \Bigr\}. \nonumber
\end{eqnarray}
\end{theorem}
\begin{proof}
For $\mbox{conv}(P^U) = Q^U$, we first show it holds for the case in which $\Cupper=\Clower+2V$, then extends to the case in which $\Cupper=\Clower+kV$ with a general $k \geq 2$ (as it is easy to verify that the conclusion holds when $k=1$ due to Theorem \ref{thm:intg_k=1}), and finally extends to the case without restrictions on $\Cupper=\Clower+kV$.

For the case in which $\Cupper=\Clower+2V$, similar to the proof for Theorem \ref{thm:intg_k=1}, we first characterize the extreme points of conv($P^U$) by utilizing Proposition \ref{prop:duc_ext_point} and then pick the corresponding $3T-1$ inequalities in $Q^U$ to be tight on each extreme point. Based on the minimum-up/-down time restrictions, as shown in Figure \ref{fig:on-off-status}, we can observe that since time $1$, the generator stays online for certain time periods (e.g., from time $1$ to time $r$, denoted as ``type-$1$'' online interval),  shut down to be offline to certain time periods, start up to be online for certain time periods (e.g., from time $m$ to time $n$, denoted as ``type-$2$'' online interval), offline for certain time periods, online for certain time period, $\cdots$, and finally start up to be online until time $T$ (e.g., from time $s$ to time $T$, denoted as ``type-$3$'' online interval).

\begin{figure}[h!]
\begin{center}
\begin{tikzpicture}
          
\draw[->] (0,0) -- (14.5,0);          

\foreach \x in {0,0.5,1.5,2,2.5,3.5,4,4.5,5.5,6,6.5,
7.5,8,8.5,9.5,10,10.5,11.5,12,12.5,13.5,14}
    \draw[shift={(\x,0)}] (0pt,4pt) -- (0pt,-4pt);

\draw (-1, 1) node {online $\rightarrow$};
\draw (-1, 0) node {offline $\rightarrow$};

\draw (0, -0.4) node {$1$};
\draw (0,1) -- (2,1);
\draw (1, -0.2) node {$\cdots$};
\draw (2,1) -- (2,0);
\draw (2, -0.4) node {$r$};

\draw (1, -0.8) node {type-$1$};

\draw (3, -0.2) node {$\cdots$};

\draw (4, -0.4) node {$m$};
\draw (4,1) -- (4,0);
\draw (4,1) -- (6,1);
\draw (5, -0.2) node {$\cdots$};
\draw (6,1) -- (6,0);
\draw (6, -0.4) node {$n$};

\draw (5, -0.8) node {type-$2$};

\draw (7, -0.2) node {$\cdots$};

\draw (8,1) -- (8,0);
\draw (8,1) -- (10,1);
\draw (9, -0.2) node {$\cdots$};
\draw (10,1) -- (10,0);

\draw (9, -0.8) node {type-$2$};

\draw (11, -0.2) node {$\cdots$};

\draw (12, -0.4) node {$s$};
\draw (12,1) -- (12,0);
\draw (12,1) -- (14,1);
\draw (13, -0.2) node {$\cdots$};
\draw (14, -0.4) node {$T$};

\draw (13, -0.8) node {type-$3$};
    
\end{tikzpicture}
\caption{online/offline status} \label{fig:on-off-status}
\end{center}
\end{figure}
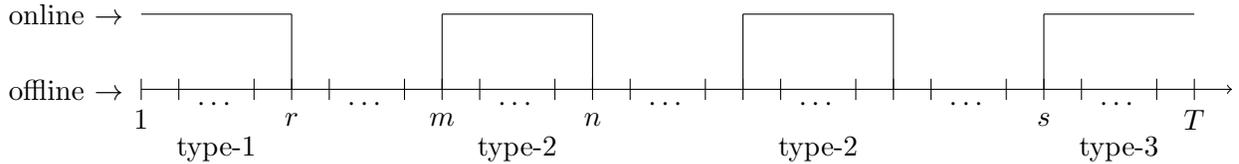

Now, we characterize the possible values of the $x$-part of each extreme point during each type of online interval, as $x_t=0$ for each $t$ in the offline intervals.

For ``type-$1$'' online interval (e.g., from time $1$ to time $r$), as shown in Figure \ref{fig:type-1}, we draw a backward scenario tree (with $r+1$ time periods/stages) with each scenario-tree node indicating the possible $x$-value at each time period, where the root node corresponds to time $r+1$ with $x$-value as $0$. In addition, for some $t$ with $x_t=\Clower+V$ then there must exist $t-1$ with $x_{t-1}=\Clower$ or $t+1$ with $x_{t+1}=\Cupper$ so that ramping-up constraints \eqref{eqn:p-ramp-up} can be tight for at least one of any two consecutive time periods. For instance, there are three possible values for $x_r$, i.e., $\Clower, \Clower+V,$ and $\Cupper$; if $x_r=\Clower$, then there are three possible values for $x_{r-1}$ (i.e., $\Clower, \Clower+V,$ and $\Cupper$); if $x_r=\Clower+V$, there is only one possible value for $x_{r-1}$ (i.e., $\Clower$) so that $x_r-x_{r-1}=V$, since $x_{r+1}-x_r\neq V$.  We can easily observe that all extreme points satisfy the possible $x$-values via the backward scenario tree in Figure \ref{fig:type-1} during the ``type-$1$'' online interval (which can be easily proved by a contradiction method that constructs two feasible points to make this point as the linear combination of these two constructed points).

\begin{figure}[h!]
\begin{center}
\begin{tikzpicture}
\draw (12,0) circle (4mm) node {$0$};
\draw [->] (11.6,0) -- (9.4, 3);
\draw [->] (11.6,0) -- (9.4, 0);
\draw [->] (11.6,0) -- (9.4, -3);
\draw (9,3) circle (4mm) node {$\Clower$};
\draw (9,0) circle (4mm) node {\tiny{$\Clower+V$}};
\draw (9,-3) circle (4mm) node {$\Cupper$};

\draw [->] (8.6,3) -- (6.4, 4.5);
\draw [->] (8.6,3) -- (6.4, 3);
\draw [->] (8.6,3) -- (6.4, 1.5);
\draw (6,4.5) circle (4mm) node {$\Clower$};
\draw (6,3) circle (4mm) node {\tiny{$\Clower+V$}};
\draw (6,1.5) circle (4mm) node {$\Cupper$};

\draw [->] (8.6,0) -- (6.4, 0);
\draw (6,0) circle (4mm) node {$\Clower$};

\draw [->] (8.6,-3) -- (6.4, -2);
\draw [->] (8.6,-3) -- (6.4, -3.5);
\draw (6,-2) circle (4mm) node {\tiny{$\Clower+V$}};
\draw (6,-3.5) circle (4mm) node {$\Cupper$};

\draw [->] (5.6,3) -- (3.4, 3);
\draw (3,3) circle (4mm) node {$\Clower$};

\draw [->] (5.6,-2) -- (3.4, -1.25);
\draw [->] (5.6,-2) -- (3.4, -2.75);
\draw (3,-1.25) circle (4mm) node {$\Clower$};
\draw (3,-2.75) circle (4mm) node {$\Cupper$};

\draw [line width=0.5pt, dashed]
      (9.5,3) to[out=110,in=5] (5.5,5) -- (5.5,1) to[out=-5,in=-110] (9.5,3) -- cycle;  
\draw (7.5, 4.5) node {$\mathcal{T}_1$};


\draw [line width=0.5pt, dashed]
      (9.5,-3) to[out=120,in=5] (5.5,-1.5) -- (5.5,-4) to[out=-5,in=-120] (9.5,-3) -- cycle;  
\draw (7.5, -2.2) node {$\mathcal{T}_2$};

\draw [->] (5.6,4.5) -- (5, 4.5);
\draw [->] (5.6,4.5) -- (5, 5);
\draw [->] (5.6,4.5) -- (5, 4);
\draw (3.9, 4.7) node {repeat};
\draw (3.9, 4.3) node {structure $\mathcal{T}_1$};

\draw [->] (5.6,0) -- (5, 0);
\draw [->] (5.6,0) -- (5, 0.5);
\draw [->] (5.6,0) -- (5, -0.5);
\draw (3.9, 0.2) node {repeat};
\draw (3.9, -0.2) node {structure $\mathcal{T}_1$};

\draw [->] (2.6,3) -- (2, 3);
\draw [->] (2.6,3) -- (2, 3.5);
\draw [->] (2.6,3) -- (2, 2.5);
\draw (.9, 3.2) node {repeat};
\draw (.9, 2.8) node {structure $\mathcal{T}_1$};

\draw [->] (2.6,-1.25) -- (2, -1.25);
\draw [->] (2.6,-1.25) -- (2, -0.75);
\draw [->] (2.6,-1.25) -- (2, -1.75);
\draw (.9, -1.05) node {repeat};
\draw (.9, -1.45) node {structure $\mathcal{T}_1$};

\draw [->] (5.6,-3.5) -- (5, -3.75);
\draw [->] (5.6,-3.5) -- (5, -3.25);
\draw (3.9, -3.3) node {repeat};
\draw (3.9, -3.7) node {structure $\mathcal{T}_2$};

\draw [->] (5.6,1.5) -- (5, 1.25);
\draw [->] (5.6,1.5) -- (5, 1.75);
\draw (3.9, 1.7) node {repeat};
\draw (3.9, 1.3) node {structure $\mathcal{T}_2$};

\draw [->] (2.6,-2.75) -- (2, -3);
\draw [->] (2.6,-2.75) -- (2, -2.5);
\draw (.9, -2.55) node {repeat};
\draw (.9, -2.95) node {structure $\mathcal{T}_2$};

\draw (12, -5) node {Time $r+1$};
\draw (9, -5) node {Time $r$};
\draw (6, -5) node {Time $r-1$};
\draw (3, -5) node {Time $r-2$};
\draw (1.5, -5) node {$\cdots$};
\draw (0, -5) node {Time $1$};

\end{tikzpicture}
\caption{type-$1$ online interval}  \label{fig:type-1}
\end{center}
\end{figure}

For ``type-$2$'' and ``type-$3$'' online intervals (e.g., from time $m$ to time $T$), we draw two forward scenario trees with each scenario-tree node indicating the possible $x$-value at each time period, where the root node corresponds to time $m-1$ with $x$-value as $0$. Meanwhile, the online/offline status of the generator is subject to the minimum-up/-down restrictions. In addition, for some $t$ with $x_t=\Clower+V$ then there must exist $t-1$ with $x_{t-1}=\Clower$ or $t+1$ with $x_{t+1}=\Cupper$ so that ramping-up constraints \eqref{eqn:p-ramp-up} can be tight for at least one of any two consecutive time periods. For instance, when starting up at time $t$, there are two possible values for $x_m$, i.e., $\Clower$ in Figure \ref{fig:type-2_3-1} and $\Vupper$ in Figure \ref{fig:type-2_3-2}; if $x_m=\Clower$, then there are three possible values for $x_{m+1}$, i.e., $0$ (subject to the minimum-up/-down time restrictions), $\Clower$, and $\Clower+V$; if $x_m=\Vupper$, then there are four possible values for $x_{m+1}$, i.e., $0$ (subject to the minimum-up/-down time restrictions), $\Clower$, $\Clower+V$, and $\Vupper+V$; if $x_m=\Vupper$ and $x_{m+1}=\Clower+V$, then there is only one possible value for $x_{m+2}$ (i.e., $\Cupper$) so that $x_{m+2}-x_{m+1}=V$, since $x_{m+1}-x_{m} \neq V$. We can easily observe that all extreme points satisfy the possible $x$-values during the ``type-$2$'' and ``type-$3$'' online intervals (which can be easily proved by a contradiction method that constructs two feasible points to make this point as the linear combination of these two constructed points).

\begin{figure}[h!]
\begin{center}
\begin{tikzpicture}
\draw (0,0) circle (4mm) node {$0$};
\draw [->] (0.4,0) -- (2.6, 0);
\draw (3,0) circle (4mm) node {$\Clower$};

\draw [->] (3.4,0) -- (5.6, 3);
\draw [->] (3.4,0) -- (5.6, 0);
\draw [->] (3.4,0) -- (5.6, -3);
\draw (6,3) circle (4mm) node {$0$};
\draw (6,0) circle (4mm) node {$\Clower$};
\draw (6,-3) circle (4mm) node {\tiny{$\Clower+V$}};

\draw [->] (6.4,3) -- (8.6, 4);
\draw [->] (6.4,3) -- (8.6, 3);
\draw [->] (6.4,3) -- (8.6, 2);
\draw (9,4) circle (4mm) node {$0$};
\draw (9,3) circle (4mm) node {$\Clower$};
\draw (9,2) circle (4mm) node {$\Vupper$};

\draw [->] (6.4,0) -- (8.6, 1);
\draw [->] (6.4,0) -- (8.6, 0);
\draw [->] (6.4,0) -- (8.6, -1);
\draw (9,1) circle (4mm) node {$0$};
\draw (9,0) circle (4mm) node {$\Clower$};
\draw (9,-1) circle (4mm) node {\tiny{$\Clower+V$}};

\draw [->] (6.4,-3) -- (8.6, -4);
\draw [->] (6.4,-3) -- (8.6, -3);
\draw [->] (6.4,-3) -- (8.6, -2);
\draw (9,-4) circle (4mm) node {$0$};
\draw (9,-3) circle (4mm) node {$\Clower$};
\draw (9,-2) circle (4mm) node {$\Cupper$};

\draw [line width=0.5pt, dashed]
      (5.5,3) to[out=70,in=180] (9.5,4.6) -- (9.5,1.5) to[out=180,in=-70] (5.5,3) -- cycle;  
\draw (7.5, 4) node {$\mathcal{T}_3$};

\draw [line width=0.5pt, dashed]
      (5.5,0) to[out=70,in=180] (9.5,1.6) -- (9.5,-1.5) to[out=180,in=-70] (5.5,0) -- cycle;  
\draw (7.5, 1) node {$\mathcal{T}_4$};

\draw [line width=0.5pt, dashed]
      (5.5,-3) to[out=70,in=180] (9.5,-1.4) -- (9.5,-4.5) to[out=180,in=-70] (5.5,-3) -- cycle;  
\draw (7.5, -2) node {$\mathcal{T}_5$};

\draw [->] (9.4,4) -- (10, 4);
\draw (10.5, 4) node {$\mathcal{T}_3$};

\draw [->] (9.4,3) -- (10, 3);
\draw (10.5, 3) node {$\mathcal{T}_4$};

\draw [->] (9.4,2) -- (10, 2);
\draw (10.5, 2) node {$\mathcal{T}_6$};

\draw [->] (9.4,1) -- (10, 1);
\draw (10.5, 1) node {$\mathcal{T}_3$};

\draw [->] (9.4,0) -- (10, 0);
\draw (10.5, 0) node {$\mathcal{T}_4$};

\draw [->] (9.4,-1) -- (10, -1);
\draw (10.5, -1) node {$\mathcal{T}_5$};

\draw [->] (9.4,-2) -- (10, -2);
\draw (10.5, -2) node {$\mathcal{T}_7$};

\draw [->] (9.4,-3) -- (10, -3);
\draw (10.5, -3) node {$\mathcal{T}_4$};

\draw [->] (9.4,-4) -- (10, -4);
\draw (10.5, -4) node {$\mathcal{T}_3$};

\draw (0, -5) node {Time $m-1$};
\draw (3, -5) node {Time $m$};
\draw (6, -5) node {Time $m+1$};
\draw (9, -5) node {Time $m+2$};
\draw (10.5, -5) node {$\cdots$};
\draw (12, -5) node {Time $T$};

\end{tikzpicture}
\caption{type-$2$ and type-$3$ online interval}  \label{fig:type-2_3-1}
\end{center}
\end{figure}

\begin{figure}[h!]
\begin{center}
\begin{tikzpicture}
\draw (0,0) circle (4mm) node {$0$};
\draw [->] (0.4,0) -- (2.6, 0);
\draw (3,0) circle (4mm) node {$\Vupper$};

\draw [->] (3.4,0) -- (5.6, 4.5);
\draw [->] (3.4,0) -- (5.6, 1.5);
\draw [->] (3.4,0) -- (5.6, -1);
\draw [->] (3.4,0) -- (5.6, -4.5);
\draw (6,4.5) circle (4mm) node {$0$};
\draw (6,1.5) circle (4mm) node {$\Clower$};
\draw (6,-1) circle (4mm) node {\tiny{$\Clower+V$}};
\draw (6,-4.5) circle (4mm) node {\tiny{$\Vupper+V$}};

\draw [->] (6.4,4.5) -- (8.6, 5.5);
\draw [->] (6.4,4.5) -- (8.6, 4.5);
\draw [->] (6.4,4.5) -- (8.6, 3.5);
\draw (9,5.5) circle (4mm) node {$0$};
\draw (9,4.5) circle (4mm) node {$\Clower$};
\draw (9,3.5) circle (4mm) node {$\Vupper$};

\draw [->] (6.4,1.5) -- (8.6, 2.5);
\draw [->] (6.4,1.5) -- (8.6, 1.5);
\draw [->] (6.4,1.5) -- (8.6, .5);
\draw (9,2.5) circle (4mm) node {$0$};
\draw (9,1.5) circle (4mm) node {$\Clower$};
\draw (9,.5) circle (4mm) node {\tiny{$\Clower+V$}};

\draw [->] (6.4,-1) -- (8.6, -1);
\draw (9,-1) circle (4mm) node {$\Cupper$};

\draw [->] (6.4,-4.5) -- (8.6, -2.5);
\draw [->] (6.4,-4.5) -- (8.6, -3.5);
\draw [->] (6.4,-4.5) -- (8.6, -4.5);
\draw [->] (6.4,-4.5) -- (8.6, -5.5);
\draw (9,-2.5) circle (4mm) node {$0$};
\draw (9,-3.5) circle (4mm) node {$\Clower$};
\draw (9,-4.5) circle (4mm) node {\tiny{$\Clower+V$}};
\draw (9,-5.5) circle (4mm) node {$\Cupper$};

\draw [->] (9.4,-1) -- (11.6, 0.5);
\draw [->] (9.4,-1) -- (11.6, -0.5);
\draw [->] (9.4,-1) -- (11.6, -1.5);
\draw [->] (9.4,-1) -- (11.6, -2.5);
\draw (12, 0.5) circle (4mm) node {$0$};
\draw (12, -0.5) circle (4mm) node {$\Clower$};
\draw (12, -1.5) circle (4mm) node {\tiny{$\Clower+V$}};
\draw (12, -2.5) circle (4mm) node {$\Cupper$};

\draw [line width=0.5pt, dashed]
      (2.5,0) to[out=80,in=180] (6.5,5) -- (6.5,-5) to[out=180,in=-80] (2.5,0) -- cycle;  
\draw (4, 2.5) node {$\mathcal{T}_6$};

\draw [line width=0.5pt, dashed]
      (8.5,-1) to[out=60,in=180] (12.5,1) -- (12.5,-3) to[out=180,in=-60] (8.5,-1) -- cycle;  
\draw (10, -0.2) node {$\mathcal{T}_7$};

\draw [->] (9.4,5.5) -- (10, 5.5);
\draw (10.5, 5.5) node {$\mathcal{T}_3$};

\draw [->] (9.4,4.5) -- (10, 4.5);
\draw (10.5, 4.5) node {$\mathcal{T}_4$};

\draw [->] (9.4,3.5) -- (10, 3.5);
\draw (10.5, 3.5) node {$\mathcal{T}_6$};

\draw [->] (9.4,2.5) -- (10, 2.5);
\draw (10.5, 2.5) node {$\mathcal{T}_3$};

\draw [->] (9.4,1.5) -- (10, 1.5);
\draw (10.5, 1.5) node {$\mathcal{T}_4$};

\draw [->] (9.4,0.5) -- (10, 0.7);
\draw (10.2, 0.7) node {$\mathcal{T}_5$};

\draw [->] (9.4,-2.5) -- (10, -2.7);
\draw (10.2, -2.7) node {$\mathcal{T}_3$};

\draw [->] (9.4,-3.5) -- (10, -3.5);
\draw (10.5, -3.5) node {$\mathcal{T}_4$};

\draw [->] (9.4,-4.5) -- (10, -4.5);
\draw (10.5, -4.5) node {$\mathcal{T}_5$};

\draw [->] (9.4,-5.5) -- (10, -5.5);
\draw (10.5, -5.5) node {$\mathcal{T}_7$};

\draw [->] (12.4,0.5) -- (13, 0.5);
\draw (13.5, 0.5) node {$\mathcal{T}_3$};

\draw [->] (12.4,-0.5) -- (13, -0.5);
\draw (13.5, -0.5) node {$\mathcal{T}_4$};

\draw [->] (12.4,-1.5) -- (13, -1.5);
\draw (13.5, -1.5) node {$\mathcal{T}_5$};

\draw [->] (12.4,-2.5) -- (13, -2.5);
\draw (13.5, -2.5) node {$\mathcal{T}_7$};

\draw (0, -6.2) node {Time $m-1$};
\draw (3, -6.2) node {Time $m$};
\draw (6, -6.2) node {Time $m+1$};
\draw (9, -6.2) node {Time $m+2$};
\draw (11.5, -6.2) node {$\cdots$};
\draw (14, -6.2) node {Time $T$};

\end{tikzpicture}
\caption{type-$2$ and type-$3$ online interval}  \label{fig:type-2_3-2}
\end{center}
\end{figure}

Now we can find the $3T-1$ tight inequalities corresponding to each extreme point as follows. For an extreme point $z=(x_1, x_2, \ldots, x_T; y_1, y_2, \ldots, y_T; u_2, \ldots, u_T)$, if $x_t \in \{0, \Clower\},$ pick $x_t \geq \Clower y_t$; else if $x_t \in \{\Vupper, \Vupper+V, \Cupper\},$ pick inequality \eqref{eqn:x_t-up-exp}; else $x_t \in \{ \Clower+V \},$ pick inequality \eqref{eqn:ru-2-exp}. Thus, here we pick $T$ tight inequalities based on the value of $x$ part of this extreme point $z$. Meanwhile, another $2T-1$ tight inequalities can be easily picked from \eqref{eqn:p-minup}-\eqref{eqn:p-lower-bound} and \eqref{eqn:u-pos} based on the values of the $y$ and $u$ parts of this extreme point $z$, since \eqref{eqn:p-minup}-\eqref{eqn:p-lower-bound} and \eqref{eqn:u-pos} already construct the convex hull of the minimum-up/-down time polytope \cite{rajan2005minimum}. Therefore, we show $\mbox{conv}(P^U) = Q^U$ holds for the case in which $k=2$.

The arguments above (including characterizing extreme points and picking the corresponding inequalities) can be easily extended to the case with a general $k \geq 2$. That is, to characterize the extreme points of conv($P$) following the conclusion in Proposition \ref{prop:duc_ext_point}, we continue to use the backward and forward scenario trees to show the possible values of the $x$ part of each extreme point. In addition, for some $t$ with $x_t=\Clower+sV$ ($s \in [1,k-1]_{\Z}$) then there must exist $t-1$ with $x_{t-1}=\Clower+(s-1)V$ or $t+1$ with $x_{t+1}=\Clower+(s+1)V$ so that ramping-up constraints \eqref{eqn:p-ramp-up} can be tight for at least one of any two consecutive time periods. Similarly, for some $t$ with $x_t=\Vupper+sV$ ($s \in [0,k-1]_{\Z}$) then there must exist $t-1$ with $x_{t-1}=\Vupper+(s-1)V$ or $0$ so that $x_t-x_{t-1}=V$. To pick the $3T-1$ tight inequalities from $Q^U$ for each extreme point of conv($P^U$), for an extreme point $(x_1, x_2, \ldots, x_T; y_1, y_2, \ldots, y_T; u_2, \ldots, u_T)$, if $x_t \in \{0, \Clower\},$ pick $x_t \geq \Clower y_t$; else if $x_t \in \{\Vupper+sV (s \in [0,k-1]_{\Z}), \Cupper\},$ pick inequality \eqref{eqn:x_t-up-exp}; else $x_t \in \{ \Clower+sV (s \in [1,k-1]_{\Z}) \},$ pick inequality \eqref{eqn:ru-2-exp}.

Furthermore, we can extend the argument to the case without restrictions on $\Cupper = \Clower+kV$ by letting $\gamma = \lfloor (\Cupper-\Clower)/V \rfloor$.

That is, following the conclusion in Proposition \ref{prop:duc_ext_point}, we continue to use the backward and forward scenario trees to characterize the extreme points of conv($P$). In addition, 
\begin{itemize}
\item[1)] for some $t$ with $x_t=\Clower+sV$ ($s \in [1,\gamma-1]_{\Z}$) then there must exist $t-1$ with $x_{t-1}=\Clower+(s-1)V$ or $t+1$ with $x_{t+1}=\Clower+(s+1)V$ so that ramping-up constraints \eqref{eqn:p-ramp-up} can be tight for at least one of any two consecutive time periods;
\item[2)] for some $t$ with $x_t=\Cupper-sV$ ($s \in [1,\gamma]_{\Z}$) then there must exist $\tilde{t}$ and $\hat{t}$ such that $\tilde{t} < t < \hat{t}$, $x_{\tilde{t}}=x_{\hat{t}}=\Cupper$, and $|x_{\bar{t}}-x_{\bar{t}+1}|=V$ for any $\tilde{t} \leq \bar{t} \leq \hat{t}-1$;
\item[3)] for some $t$ with $x_t=\Vupper+sV$ ($s \in [0,k-1]_{\Z}$) then there must exist $t-1$ with $x_{t-1}=\Vupper+(s-1)V$ or $0$ so that $x_t-x_{t-1}=V$.
\end{itemize}
To pick the $3T-1$ tight inequalities from $Q^U$ for each extreme point of conv($P^U$), for an extreme point $(x_1, x_2, \ldots, x_T; y_1, y_2, \ldots, y_T; u_2, \ldots, u_T)$, if $x_t \in \{0, \Clower\},$ pick $x_t \geq \Clower y_t$; else if $x_t \in \{\Vupper+sV (s \in [0,\gamma-1]_{\Z}), \Cupper\},$ pick inequality \eqref{eqn:x_t-up-exp}; else $x_t \in \{ \Clower+sV (s \in [1,\gamma]_{\Z}), \Cupper-sV (s \in [1,\gamma]_{\Z}) \},$ pick inequality \eqref{eqn:ru-2-exp}. Therefore, we show $\mbox{conv}(P^U) = Q^U$ holds in general. Due to the similarity, we omit the proof for the conclusion $\mbox{conv}(P^D) = Q^D$.
\end{proof}


\section{Conclusions} \label{sec:conclusion}
We presented the convex hull results of the unit commitment polytope under different settings. In particular, the convex hull descriptions for two special cases, i.e., $V=\Cupper-\Clower$ and $V=(\Cupper-\Clower)/2$ with general $T$ time periods are firstly provided for the integrated minimum-up/-down time and ramping polytope. Then the convex hull descriptions for the integrated minimum-up/-down time and ramping-up (or-down) polytope are provided under the most general setting.

\baselineskip=12pt
\bibliographystyle{plain}
\bibliography{uc_ramp}

\begin{thebibliography}{1}

\bibitem{lee2004min}
J.~Lee, J.~Leung, and F.~Margot.
\newblock Min-up/min-down polytopes.
\newblock {\em Discrete Optimization}, 1(1):77--85, 2004.

\bibitem{pan2016polynomial1}
K.~Pan, K.~Zhou, and Y.~Guan.
\newblock Polynomial time algorithms and extended formulations for unit
  commitment problems.
\newblock {\em Working Paper}, [Online] \url{https://arxiv.org/abs/1608.00042},
  2016.

\bibitem{rajan2005minimum}
D.~Rajan and S.~Takriti.
\newblock Minimum up/down polytopes of the unit commitment problem with
  start-up costs.
\newblock {\em IBM Research Report RC23628}, Jun. 2005.

\end{thebibliography}

\end{document}